\documentclass[american]{amsart}
\usepackage[T1]{fontenc}
\usepackage[latin9]{inputenc}
\usepackage{amsthm}
\usepackage{setspace}
\usepackage{amssymb}
\numberwithin{equation}{section} 
\numberwithin{figure}{section} 
\theoremstyle{plain}
  \theoremstyle{plain}
  \newtheorem*{thm*}{Theorem}
\newtheorem{thm}{Theorem}[section]
  \theoremstyle{plain}
  \newtheorem{cor}[thm]{Corollary}
  \theoremstyle{definition}
  \newtheorem{problem}[thm]{Problem}
  \theoremstyle{plain}
  \newtheorem{prop}[thm]{Proposition}
  \theoremstyle{definition}
  \newtheorem{defn}[thm]{Definition}
  \theoremstyle{plain}
  \newtheorem{lem}[thm]{Lemma}
  \theoremstyle{plain}
  \newtheorem*{lem*}{Lemma}
  \theoremstyle{remark}
  \newtheorem{claim}[thm]{Claim}

\DeclareMathOperator{\id}{id}

\usepackage{epsfig}
\usepackage{color}

\usepackage{babel}

\begin{document}

\title{Non expansive directions for $\mathbb{Z}^{2}$ actions}

\author{Michael Hochman}

\address{Department of Mathematics, Princeton University, Princeton, NJ 08544.}

\email{hochman@math.princeton.edu}


\thanks{This research was partially supported by the NSF under agreement
No. DMS-0635607.}
\begin{abstract}
We show that any direction in the plane occurs as the unique non-expansive
direction of a $\mathbb{Z}^{2}$ action, answering a question of Boyle
and Lind. In the case of rational directions, the subaction obtained
is non-trivial. We also establish that a cellular automaton can have
zero Lyapunov exponents and at the same time act sensitively; and
more generally, for any positive real $\theta$ there is a cellular
automaton acting on an appropriate subshift with $\lambda^{+}=-\lambda^{-}=\theta$.
\end{abstract}
\maketitle

\section{\label{sec:Introduction}Introduction}

Consider a $\mathbb{Z}^{2}$ action $(X,T)$ on a compact metric space.
Let $\ell$ be a line in the plane and let $\ell^{r}$ denote the
set of points within distance $r$ of $\ell$. Then $\ell$ is said
to be an expansive line if there exist $r>0$ and $\delta>0$ such
that, for all $x,y\in X$, \[
d(T^{u}x,T^{u}y)<\delta\mbox{ for all }u\in\ell^{r}\cap\mathbb{Z}^{2}\quad\implies\quad x=y\]
Expansiveness only depends on the direction of the line and not the
line itself, so we may speak of expansive and non-expansive directions.
Note that if $\ell$ contains an integer point $u\in\ell\cap\mathbb{Z}^{2}$
then expansiveness of $\ell$ is equivalent to expansiveness of the
map $T^{u}:X\rightarrow X$, but for $\ell$ with irrational slope
there is no such interpretation.%
\footnote{The first thing one tries is to go to a continuous time analog with
flow $\{\sigma^{t}\}_{t\in\mathbb{R}}$, and set $\varphi=\sigma^{\theta}$;
but this does not work since when one goes back to a discrete action
one loses expansiveness.%
}

Expansive and non-expansive directions were defined by Boyle and Lind
in \cite{BoyleLind97}, where they were used as a tool in the study
of the directional dynamics of an action. Many properties of the dynamics
of subactions $T^{u}$ vary nicely withing  connected components of
the set of expansive directions. For example, within such a component
the entropy of subactions varies linearly. Certain properties are
constant within expansive components: for example, if $T^{u}$ acts
as a shift of finite type then so does $T^{v}$ as long as the directions
$u,v$ are in the same expansive component. 

One of the basic questions that arose in \cite{BoyleLind97} was to
understand what sets can occur as the set of non-expansive directions.
It was shown that this set is closed and, when the phase space is
infinite, non-empty; and furthermore if $C$ is any closed set of
directions of cardinality $|C|\geq2$, then it is the set of non-expansive
directions for some action. 

It has been an open problem for some time to determine which direction
can occur as the unique non-expansive direction in a non-trivial way.
If one begins with an expansive $\mathbb{Z}$-action $(X,T)$ and
extends it formally to the $\mathbb{Z}^{2}$ action $(X,\left\langle T,\id_{X}\right\rangle )$
generated by $T$ and the identity map, then one obtains an action
whose unique non-expansive direction is the vertical one; but the
action in that direction is trivial. One can construct similarly trivial
examples in which an arbitrary rational direction is the only non-expansive
one. However, attempts have not succeeded in producing non-trivial
examples for rational directions (a proposed example in \cite{Madden2000}
turned out to be flawed, see \cite{Boyle2008}), or any examples at
all of actions with a single irrational non-expansive direction.

In this paper we resolve this problem as follows:
\begin{thm*}
\label{thm:main}For every direction $\ell$ in the plane there is
an expansive $\mathbb{Z}^{2}$ action whose unique non-expansive direction
is $\ell$. In the case $\ell$ has rational slope, the corresponding
subaction is non-trivial, i.e. none of its elements act as the identity.\end{thm*}
\begin{cor}
A set of directions occurs as the set of non-expansive directions
for an expansive $\mathbb{Z}^{2}$-action if and only if it is closed
and non-empty.
\end{cor}
This follows by combining the theorem with Boyle and Lind's result
for sets of size $\geq2$, but can also be derived directly from our
construction by taking unions of the systems it provides. The Boyle-Lind
examples are also unions of a similar sort. Thus the constructions
we have are quite degenerate, in the sense that they decompose into
subsystems with small sets of non-expansive directions. The following
question is therefore natural:
\begin{problem}
Can every nonempty closed set of directions occur as the non-expansive
directions of $\mathbb{Z}^{2}$-action that is transitive/minimal/supports
a global ergodic measure?
\end{problem}
Another consequence of our construction is:
\begin{prop}
\label{pro:Non-sensitive-example}\label{cor:lyapunov}There exists
a cellular automaton $f$ such that, for every $t$, there is a subshift
$X_{t}$ on which $f$ acts as an automorphism without equicontinuity
points and such that the Lyapunov exponents are $\lambda^{+}=t,\lambda^{-}=-t$.
\end{prop}
This answers to a question of Bressaud and Tisseur \cite{BressaudTisseur2007}
in the special case $t=0$, showing that even for cellular automata
whose action is sensitive to initial conditions, information can propagate
unboundedly, but at a sublinear rate, i.e. with zero speed.

Theorem \ref{thm:main} and the last proposition are related as follows.
Suppose we wish to realize a line $\ell$ as the unique non-expansive
direction of a $\mathbb{Z}^{2}$-action. We shall do so on a zero-dimensional
phase space. In this case we may fix an expansive direction with rational
slope, and choose another rational direction so that together the
actions in these directions generate the full action (or a finite-index
subgroup of it). Using expansiveness we may re-code and identify the
first direction with the shift on some symbolic space $X$, and the
second direction with an automorphism of $X$. Thus the problem has
been reduced to one of constructing an appropriate shift space and
an automorphism of it; this is the same setting as is studied in the
theory of cellular automata.

This reduction highlights an interesting aspect of the problem. Each
automorphism is given by a block code. There are only countably many
of these, but there are uncountably many directions (or values for
Lyapunov exponents). Thus if one is to construct examples of automorphisms
which realize any given direction as the unique non-expansive one
for the generated action (or if one wants to construct CA with arbitrary
Lyapunov exponents), then the direction (or exponents) must be encoded
at least in part in the subshift rather than the automorphism. We
shall make this encoding quite explicit, effectively designing the
automorphism as an interpreter and using the subshift as a program
controlling the action of the automorphism. 

Our strategy will be to construct an automorphism that, roughly speaking,
performs a sequence of shifts on the underlying space at a rate that
is encoded in the subshift it is acting on; this rate is what will
determine the slope of the non-expansive direction. For example, taking
the full shift as our space, $\sigma$ as the shift and the automorphism
$\varphi=\sigma^{n}$, we see that in the generated $\mathbb{Z}^{2}$-action
every direction is expansive except the line $ny+x=0$. Here $\varphi$
shifts at a rate of $n$ symbols per unit time. We would like to control
this rate so as to make it an arbitrary real number $\theta$. The
implementation of this simple idea, however, is rather involved. Our
solution relies on a property that has been called \emph{intrinsic
universality} in the cellular automata literature (e.g. \cite{AlbertCulik1987}).
This means that there are automorphisms which, when restricted to
an appropriate subshift, can simulate any other automorphism up to
a temporal and spacial rescaling. We shall use an infinite hierarchy
of such automorphisms, each of which simulates the next, and such
that each level in the hierarchy performs a shift on the underlying
space at a fixed rate (this construction is somewhat reminiscent of
Gacs' error-correcting automata \cite{Gacs2001}). The sum of these
rates will determine the overall shift and the non-expansive direction,
and we will control these rates by encoding them into the shift space. 

We shall mostly use standard definitions and notation, which can be
found e.g. in \cite{Walters82}. We denote by $\sigma$ the shift
map on symbol spaces, and for a point $x\in\Sigma^{\mathbb{Z}}$ we
denote its coordinates by $x_{i}$. For a symbol $a$ we write $a^{n}$
for the $n$-fold concatenation of $a$. An automorphism $\varphi$
of a subshift $Y$ is given by a block code, and we shall say that
$\varphi$ has range $r$ if the block code acts on a $r$-neighborhood
$[-r,r]$ (this is also sometimes called the radius of $\varphi$
or its window width). The notation $O_{N}(1)$ denotes a constant
depending only on $N$. 

The rest of this paper is organized as follows. In the next section
we introduce a sufficient condition for the action generated by a
shift-automorphism and the shift to have a unique non-expansive direction.
Sections \ref{sec:Main-construction} outlines the main construction,
section \ref{sec:Implementation-details} supplies further details
of the implementation. Section \ref{sec:Realizing-unique-non-expansive-directions}
applies the construction to prove the main theorem. Section \ref{sec:Lyapunov-exponents}
discusses the relation and applications to Lyapunov exponents. Finally,
in section \ref{sec:alternative-construction} we present an simpler,
alternative construction of a system with a unique, rational non-expansive
direction.

\emph{Acknowledgment: }I would like to thank Doug Lind for some very
interesting discussions and for his permission to include the example
in section \ref{sec:alternative-construction}. This work was done
in the fall of 2008 during the special semester on additive combinatorics
and ergodic theory at MSRI, and I would like to thank the organizers
and hosts for that stimulating event.

\section{\label{sec:Prediction-shapes}Prediction shapes}

In this section we define a device that quantifies quantify propagation
of uncertainty under iteration of an automorphism. This device is
related to Shereshevsky's notion of Lyapunov exponents for cellular
automata \cite{Shereshevsky92}, although there are a number of differences.
First, we are interested in the propagation of uncertainty both forward
and backward in time, although one can easily modify our definitions
so that they are one sided, and apply to endomorphisms as well. More
importantly we measure uncertainty by fixing a finite block, rather
than a one-sided infinite ray (in Shershevsky's case one fixes a leaf
of the stable or unstable foliations). We shall discuss the relation
to Lyapunov exponents further in section \ref{sec:Lyapunov-exponents}.
\begin{defn}
\label{def:prediction-shapes}Let $Y$ be a subshift and $\varphi$
an automorphism of $Y$. A convex, open subset $\Lambda\subseteq\mathbb{R}^{2}$
is called a \emph{prediction shape }for $\varphi|_{Y}$ if $(0,1)\times\{0\}\subseteq\Lambda$
and for every compact set $\Lambda_{0}\subseteq\Lambda$ and for all
large enough $n$, if $y,z\in Y$ satisfy \[
y|_{[-n,n]}=z|_{[-n,n]}\]
then\[
(\varphi^{t}y)_{i}=(\varphi^{t}z)_{i}\mbox{ for all }(i,t)\in n\Lambda_{0}\cap\mathbb{Z}^{2}\]
where $n\Lambda_{0}=\{n\cdot u\,:\, u\in\Lambda_{0}\}$.
\end{defn}
From the definition it is clear that the increasing union of prediction
shapes is a prediction shape.

If $\varphi,\varphi^{-1}$ have range $r$ then the diamond shaped
region with vertices at $(-1,0)$ and $(1,0)$ and with sides of slope
$\pm1/r$ is a prediction shape for $\varphi|_{Y}$. This bound derives
from information about the block-code and one may sometimes get more
from the block code, but in general the prediction shapes for $\varphi|_{Y}$
depend non-trivially on $Y$. For example, if $Y$ is a finite union
of periodic orbits then for large enough $n$ the restriction $y|_{[-n,n]}$
determines $y\in Y$; therefore $\mathbb{R}^{2}$ is a prediction
region. On the other hand, for infinite subshifts it is easy to see
that no prediction shape contains $[-1,1]\times\{0\}$ in its interior. 

Our application of this notion is the following simple observation.
\begin{thm}
\label{thm:speeds-and-non-expansive-directions}Let $Y$ be a subshift
with automorphism $\varphi$. Let $\ell$ be a line through the origin
distinct from the $x$-axis, and suppose that \[
\Lambda=\ell^{1}=\{u\in\mathbb{R}^{2}\,:\, d(u,\ell)<1\}\]
is a prediction shape for $\varphi|_{Y}$. Then every direction except
$\ell$ is an expansive direction for the $\mathbb{Z}^{2}$-system
$(Y,\left\langle \sigma,\varphi\right\rangle )$. \end{thm}
\begin{proof}
Fix a line $\ell'$ through the origin in a different direction from
$\ell$. We must show that if $r$ is large enough then $y|_{(\ell')^{r}}$
determines $y$. For this it suffices to show that there is an $r$
such that $y|_{(\ell')^{r}}$ determines $y|_{(\ell')^{r+1}}$. 

Since the slopes of $\ell,\ell'$ are different, it is easy to see
that there is an $\varepsilon>0$ and a compact set $\Lambda_{0}\subseteq\Lambda$
containing the origin with the property that $(\ell')^{1+\varepsilon}\subseteq(\ell')^{1}+\Lambda_{0}$.
The desired conclusion now follows from the fact that $\Lambda$ is
a prediction shape. See figure \ref{fig:prediction-shape}.
\end{proof}
\begin{figure}
\input{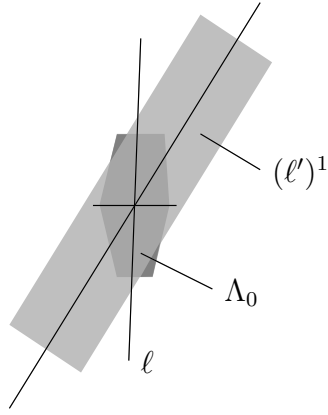}

\caption{By shifting the center of the dark region $\Lambda_{0}$ along the
line $\ell'$, one covers $(\ell')^{1+\varepsilon}$ (the area swept
out by the darkened corners of $\Lambda_{0}$).\label{fig:prediction-shape}}

\end{figure}

\begin{cor}
\label{cor:condition-for-unique-non-expansive-direction}Under the
assumptions of the theorem, if $Y$ is infinite then $\ell$ is the
unique non-expansive direction for $(Y,\left\langle \sigma,\varphi\right\rangle )$.\end{cor}
\begin{proof}
This follows from the theorem and the fact that every $\mathbb{Z}^{2}$-action
on an infinite space must have at least one non-expansive direction
\cite{BoyleLind97}.
\end{proof}

\section{\label{sec:Main-construction}Main construction}

In this section and the next we construct a subshift $X$ over an
appropriate alphabet, and define a pair of endomorphisms $\pi$ and
$\widehat{\pi}$ of $X$ by specifying block codes for them. This
section describes their properties and outlines the construction.
Some further details appear in the next section.

By construction, $X,\pi$ and $\widehat{\pi}$ will satisfy the following
properties. First, the range of $\pi$ and $\widehat{\pi}$ will be
$1$, meaning that $\pi(x)_{0}$ depends only on $x_{-1},x_{0},x_{1}$
and similarly for $\widehat{\pi}$.

Second, suppose we are given the following parameters:
\begin{itemize}
\item An integer $N$.
\item A subshift $Y\subseteq\{1,\ldots,N\}^{\mathbb{Z}}$.
\item Block codes of range $1$ defining inverse automorphisms $\varphi$,$\varphi^{-1}$
of $Y$.
\item An integer $B\geq1$ ({}``Block length''), which is sufficiently
large with respect to $N,\varphi$.
\item An integer $W\geq1$ ({}``Wait time'').
\item An integer $D$ ({}``Displacement''), which may be positive or negative,
indicating displacement to the right or left respectively.
\end{itemize}
Then there is a subshift \[
X'=X'(Y,N,\varphi,B,W,D)\subseteq X\]
such that \[
\widehat{\pi}|_{X'}=(\pi|_{X'})^{-1}\]
and an integer \begin{equation}
T=(B+O_{N,\varphi}(1))(1+W+|D|)\label{eq:T-bound}\end{equation}
so that $(X',\pi)$ and $(Y,\varphi)$ are related in the following
manner. Each configuration of $x\in X'$ breaks into blocks of length
$B$ each representing one symbol from the alphabet $\{1,\ldots,N\}$
of $Y$, and thus $x\in X'$ encodes a point $y\in Y$. With this
interpretation of $x$, the endomorphism $\pi^{T}$ acts on $x$ in
the same manner that $\sigma^{D}\circ\varphi$ acts on $Y$, i.e.
it applies $\varphi$ to the encoded sequence $y$ without altering
the block structure, resulting in a new point $x'\in X$ encoding
$\varphi(y)$, and then shifts each block of $x'$ a distance of $D$
blocks, i.e. $D\cdot B$ symbols, mimicking the action of $\sigma^{D}$
on $\varphi(y)$. More precisely, there is an isomorphism\begin{equation}
(X',\sigma,\pi)\cong(Y\times\{0,\ldots,B-1\}\times\{0,\ldots,T-1\},\sigma_{B},\varphi_{T})\label{eq:simulation-isom}\end{equation}
where $\sigma_{B},\varphi_{T}$ are the suspension maps defined by\[
\begin{array}{ccccccccc}
\sigma_{B}(y,b,t) & = & ( & \sigma^{\{b=0\}}y & , & b+1\bmod B & , & t & )\\
\varphi_{T}(y,b,t) & = & ( & (\sigma^{D}\varphi)^{\{t=0\}}y & , & b & , & t+1\bmod T & )\end{array}\]
Here we have denoted $\{b=0\}=\delta_{b0}$and $\{t=0\}=\delta_{t0}$.
See figure \ref{fig:block-layers-fig}

\begin{figure}
\input{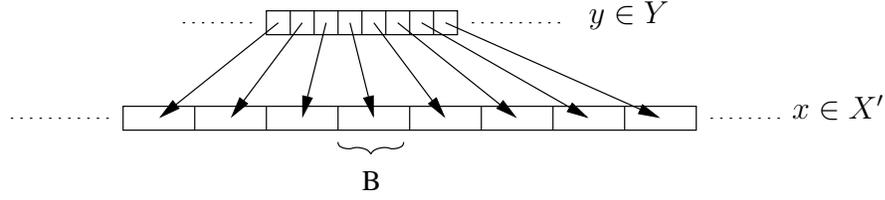}

\caption{\label{fig:block-layers-fig}Encoding a point in $y$ as a sequence
of blocks.}

\end{figure}

What we have required of $\pi,\widehat{\pi}$ is very similar to what
is called intrinsic universality, which has been studied in the CA
literature (e.g.\cite{AlbertCulik1987}), although the additional
shift by $D$ blocks is special to our construction. Such endomorphisms
have been constructed many times, as well as automorphisms \cite{MoritaHarao89}.
However, we have not found a reference that satisfies all of our requirements
exactly, and for this reason and in the interest of completeness we
provide an outline of the construction details. In this section we
give an overview; in the next we give some of the finer details. However,
the properties above are all we shall use about $\pi,\widehat{\pi}$
and one may prefer at this point skip ahead to section \ref{sec:Realizing-unique-non-expansive-directions}
where we prove theorem \ref{thm:main}.

We construct $\pi,\widehat{\pi}$ in a manner independent of the parameters
above; $\pi,\widehat{\pi}$ will operate as {}``interpreters'',
and the other parameters will be encoded in the configurations of
$X'$ so as to influence the way in which $\pi$ acts. We shall eventually
set $X=\overline{\cup X'}$, where the union ranges over all choices
of parameters. Notice that although $X$ is larger than $\cup X'$,
nonetheless $\widehat{\pi}=\pi^{-1}$ on $X$, because this is true
on each $X'$ (and $\pi,\widehat{\pi}$ are given by the same block
code on all of them).

For the construction we shall assume that the parameters $Y,N,\varphi,B,W,D$
are given, and describe the block codes for $\pi,\widehat{\pi}$ in
a way that is independent of the parameters, and a subshift $X'$
that depends on them.

The \textbf{alphabet }of $X'$ consists of quadruples of symbols,
which we denote $(b,p,s,d)$: here $b$ stands for Block structure,
$p$ for Program, $s$ for State and $d$ for Data. The projection
of a sequence onto each of these coordinates are layers: for example
the sequence of data components is the data layer.

The symbols used in the data layer will include the symbols $0,1$,
which we call \textbf{bits}. We shall represent each symbol of $Y$
by a sequence of $\left\lceil \log N\right\rceil $ bits followed
by an appropriate terminating symbol; together we call such a sequence
a \textbf{word}. 

The \textbf{Data layer }consists of two words, representing symbols
from $Y$, starting at the left side of the block; the remaining space
to the right of these words is filled with {}``blank'' symbols.
The first word in the pair is interpreted as the symbol of $Y$ currently
represented by the block, and the second as the symbol represented
by the block in the previous cycle. There will be times when the data
will be in a corrupt state, but even then it will be possible to recover
the uncorrupted current and previous states from the other layers;
see below.

The \textbf{Block layer }is a periodic sequence whose period is $B$
and is not modified by $\pi$. At this stage we may take the block
layer to be a periodic concatenation of the string $1\,0^{B-1}$,
and we shall call such a sequence, and also the indices it occupies,
simply a block (later on we will add more information to this layer;
see section \ref{sec:Implementation-details}). In the identification
$X'\cong Y\times\{0,\ldots,B-1\}\times\{0,\ldots,T-1\}$, the second
component in the image of $x\in X'$ will be determined by the residue
class mod $B$ of the position of $1$'s in the block layer of $x$.
We remark that since $B$ can be arbitrarily large, $X$ will contain
also points whose block layer consists of all $0$'s, or of a single
$1$ surrounded by $0$'s; but as we shall see these will not cause
a problem.

The \textbf{Program layer }is also periodic with period $B$ and is
not modified by $\pi$. The repeated sequence, which we call simply
the Program, is constant throughout $X'$. The program begins at the
left side with an encoding of the block codes of $\varphi$ and $\varphi^{-1}$.
This takes the form of a sequence of 5-tuples of words, representing
5-tuples of symbols from $Y$. Each 5-tuple represents an input (3
$Y$-symbols) and the corresponding output symbol of $\varphi$ and
$\varphi^{-1}$ (recall that both are assumed to have range $1$).
We separate these 5-tuples from each other with some special symbol,
and terminate the sequence of 5-tuples with another special symbol.
Next, the program layer contains the parameters $W,D$ encoded as
contiguous sequences of $1$'s, either $W$ or $D$ in number, followed
by terminating symbols. The remainder of the program layer is filled
with blanks.

Finally, the \textbf{State layer }contains auxiliary information used
to interpret the Program layer and use it to update the Data layer.
We call a sequence of state symbols corresponding to a Block simply
a State-block. The state-blocks in different blocks typically differ
from each other, since they depend on the data in the block and the
neighboring blocks, but they will be synchronized in the following
sense: there is a special state-block called the Synchronized State,
so that once (and only once) every $T$ applications of $\pi$, the
synchronized state-block appears in all the blocks of a configuration
$x\in X'$. We shall call an $x\in X'$ with all blocks synchronized
a synchronized configuration. It is during such a time that the data
layer is guaranteed to represent correctly the current and previous
symbols. Thus in the identification $X\cong Y\times\{0,\ldots,B-1\}\times\{0,\ldots,T-1\}$,
the third component of the image of $x$ is the number $t$ of applications
of $\pi^{-1}$ needed to bring $x$ to a synchronized configurations,
and the first component is the current symbol in the data layer of
of $\pi^{-t}x$.

The layout of a block is depicted in figure \ref{fig:block-layout}

\begin{figure}
\input{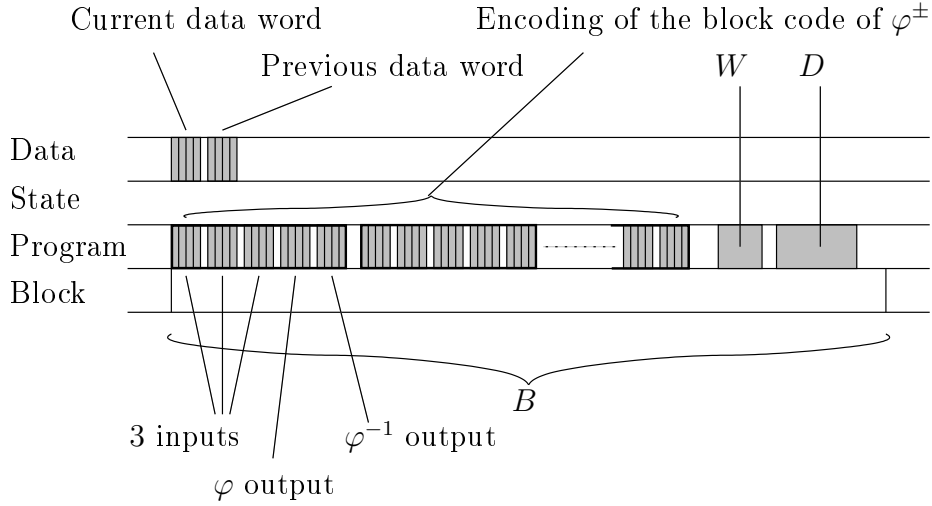}

\caption{The arrangement of information inside a block.\label{fig:block-layout}}

\end{figure}

Applying $\pi$ repeatedly to a synchronized configuration $x\in X'$
leads to the following sequence of events, which we call a Cycle (figure
\ref{fig:lookup-cycle}): 
\begin{enumerate}
\item Each block {}``transmits'' its current data word to the neighboring
blocks on its left and right, and receives the same information from
them. At the end of this stage, the state contains two words, $a_{L},a_{R}$,
representing the current $Y$-symbol encoded in the neighboring blocks
to the left and right of the current block, respectively.

This stage takes $B+O_{N}(1)$ applications of $\pi$ to complete
(each bit moves a distance of $B$, taking $B$ time steps, plus a
constant amount of time needed too coordinate the transmission which
depends only on the number of bits being transmitted, which depends
on $N$).

\item The two words from the data layer -- the current and previous $Y$-symbols
-- are copied to the state layer and simultaneously deleted from the
data layer. If $a_{C}$ is the current symbol and $a_{P}$ the previous
symbol of the block, then the state space now contains the 4-tuple
$a_{L}a_{C}a_{R}a_{P}$ of $Y$-symbols.

This stage takes $O_{N}(1)$ applications of $\pi$ to complete.

\item We now enter a loop in the course of which the 4-tuple in the state
layer is translated to the right, stopping opposite each 5-tuple in
the program layer. 

\begin{enumerate}
\item For each 5-tuple it checks if the triple of words $a_{L}a_{C}a_{R}$
matches the input-triple in the program layer. 
\item \label{enu:When-a-match-is-found}When a match is found the words
$a_{L},a_{R},$ and $a_{P}$ in the state layer are erased, and the
corresponding output word $b=\varphi_{0}(a_{L}a_{C}a_{R})$, which
is encoded in the data layer, is copied to the state layer. The state
layer now contains the current $Y$-symbol $a_{C}$ and future $Y$-symbol
$b$. 
\item The comparisons continue also after a match is found. The implementation
will be such that each comparison takes the same number of steps,
whether or not a match is found. A match will be found exactly once.
\end{enumerate}
Since the comparisons continue at the same rate after a match is made,
the number of applications of $\pi$ in this stage is independent
of the configuration, and this stage ends at the same point in the
cycle for all blocks in a configuration. The time for this step is
$O_{N,\varphi}(1)$.

\item The current and future $Y$-symbols are translated back through the
state layer to the left end of the block and transferred to the data
layer: $b$ is copied to the {}``current'' slot and $a_{C}$ to
the {}``previous'' slot, and they are simultaneously deleted from
the state layer.

This step takes $O_{N,\varphi}(1)$ applications of $\pi$ to complete.

\item The data layer is shifted $|D|$ blocks to the left or right, according
to the sign of $D$. For this, the sign is first determined, and then
a loop is performed during each iteration of which the data layer
is shifted by one block length.

This step takes $(B+O_{N}(1))\left|D\right|$ applications of $\pi$;
the $B$ term corresponds to actually transporting the data. The $O_{N}(1)$
term is the overhead required each cycle. We give further details
in the next section.

\item The state layer {}``mock-shifts'' the data layer $W$ more times,
meaning that the state goes through a cycle which takes the same amount
of time as shifting it one block, but doesn't result in such a shift
taking place.

This step takes $(B+O_{N}(1))W$ applications of $\pi$.

\item All blocks return to the synchronized state.

This takes $O_{N}(1)$ applications of $\pi$.

\end{enumerate}
Part of a cycle is depicted in figure \ref{fig:lookup-cycle}.

\begin{figure}

\hspace{-3cm}\input{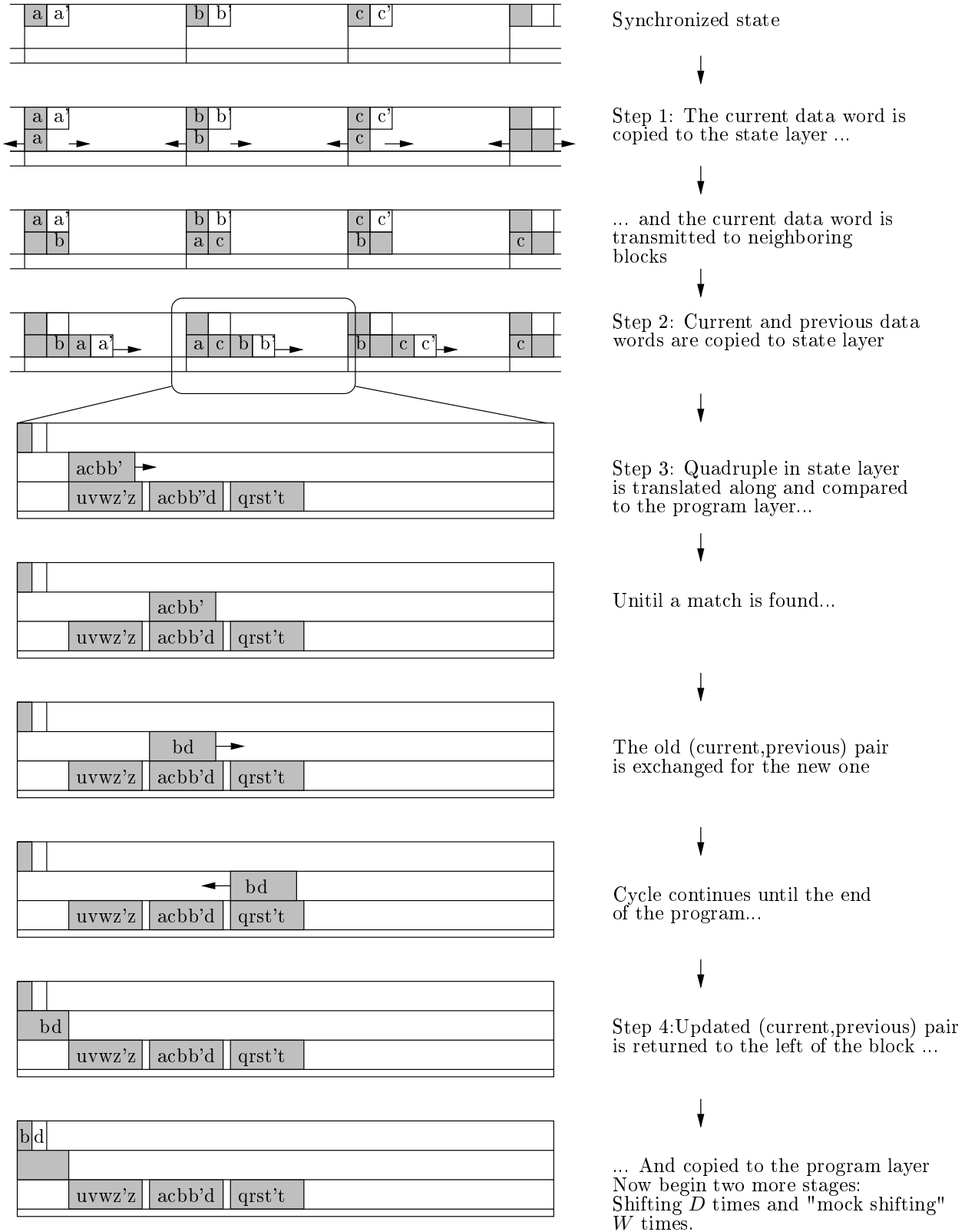}\caption{Part of a cycle, including simulation (but not the shifting).\label{fig:lookup-cycle}}

\end{figure}

Given the parameters $Y,\varphi$ etc., let $X''\subseteq X$ be the
set of synchronized configurations whose program layer corresponds
to the given parameters and whose sequence of symbols encoded in the
data layer correspond to points in $y\in Y$ and $\varphi^{-1}(y)$.
We then set $X'=\cup_{t=0}^{T-1}\pi^{t}X''$, where $T$ is the length
of one cycle.

It is clear from this description that $\pi^{T}$ simulates $\varphi$,
where $T$ is the length of a cycle. The isomorphism (\ref{eq:simulation-isom})
is given by\[
x\mapsto(y,b,t)\]
where $b$ is the least non-negative integer such that $-b$ is the
coordinate of the beginning of a block; $t$ is the number of applications
of $\pi^{-1}$ to $x$ required to bring $x$ to a synchronized configuration;
and $y\in Y$ is the sequence encoded by $\pi^{-t}(x)$, where $y_{0}$
is the symbol coded in the block to which $0$ belongs in $\pi^{-t}(x)$.
The length of a cycle is given by (\ref{eq:T-bound}).

Note that the alphabet of $X'$, and therefore of $X$, is independent
of the parameters, and in particular of $N$. Thus we can simulate
systems $(Y,\sigma,\varphi)$ on arbitrarily large alphabets by subsystems
$X'\subseteq X$ . The upshot is that in order to do so the parameter
$B$ must be large enough that the blocks can encode the parameters
as described above.

The implementation details of $\pi$ are very similar to those involved
in constructing a universal Turing machine, and are completely standard,
with one exception: generally Turing machines are not reversible,
yet we want $\pi$ to act invertibly on $X'$ and for $\widehat{\pi}=\pi^{-1}$
to have range $1$. With care this can be done. Notice that the only
place where information is deleted in the scheme above is in the stage
(\ref{enu:When-a-match-is-found}), where the 5-tuple in the state
layer matches the 4-tuple in the program layer. At this point certain
information is erased from and added to the state layer, both are
present in the program layer; this allows the process to be reversed
locally (this is the reason $\varphi^{-1}$ is encoded in the program
layer; notice that it is not used explicitly in the definition of
$\pi$). The other steps -- namely, the transferring of bits from
one place to another, etc. -- can be done invertibly with inverse
having range $1$. Thus we have achieved our stated goal. The bound
(\ref{eq:T-bound}) for $T$ follows easily from the construction.

\section{\label{sec:Implementation-details}Implementation details}

This section outlines the realization of the automorphism $\pi$ described
in the previous section. It is provided for completeness and readers
may prefer to skip ahead to the next section where the main construction
is undertaken. CA simulating other CA have been constructed a number
of times in the literature, e.g. \cite{AlbertCulik1987}, and are
similar to universal Turing machines. The only new ingredient here
is a careful analysis of certain aspects of the time complexity of
the simulation, and our emphasis on invertibility. Although both aspects
have been addressed in the literature, it is easier to indicate the
construction than to explain how to modify existing ones to meet our
specific needs.

As a complete implementation of $\pi$ would be a very lengthy undertaking,
we describe only the part of the implementation responsible for the
first stage of the cycle, in which the current data words are transmitted
between immediate neighboring blocks. In this stage we already encounter
the main ideas needed to complete the rest, and we provide a few hints
about the other stages.

The symbols of the state layer represent sets populated by \emph{agents}.
An agent is a finite state automaton. Each agent in each state-layer
cell will, with each application of $\pi$, perform one or more of
the following operations: (1) modify the symbols in the data layer,
(2) update its internal state, (3) move one cell left or right. By
design, not more than one agent per cell will attempt to modify the
data layer, so no conflict will arise. The nature of the operation
that an agent performs is determined by the other contents of its
cell prior to the operation, including its own previous internal state,
the states of the other agents in its cell, and the symbols in the
data, program and block layers. 

There are two types of agents:
\begin{itemize}
\item The \emph{main agent}. There is one such agent per block, and it is
the only agent capable of modifying the data layer. 
\item \emph{Data agents}. Used to store and transport data. We allow several
types of data agents, which play slightly different roles, but they
overall behavior is the same. 
\end{itemize}
We shall also add new symbols to the block layer. We call these \emph{roadsigns}.
Their role is to signal some event to the agents at that cell.

Note that, since the laphabet of $X$ may not depend on the parameters
$N,Y,\varphi,B,D,W$, we may introduce only a finite number of agents
and new symbols (roadsigns); but we may arrange them as we wish. This
allows us a great deal of flexibility in programming the agents and
providing them external cues to modify their behavior.

The data agents role is to store and transport a single bit of data.
Their internal state consists of a motion symbol ({}``left'', {}``right'',
or {}``stationary''); and a data symbol ({}``0'', {}``1'' or
{}``empty''). At each step, a data agent updates its state based
on roadsigns in the block layer and instructions from the main agent,
if it occupies the same cell as the data agent. Then the data agent
takes one step left or right or stays in its current cell, according
to its motion symbol.

The main agent acts as coordinator, issuing instructions to data agents
and modifying the data layer. Its operation is more complex, and it
is our goal to describe some of it in detail.

At the beginning of each cycle all agents are arranged as follows.
The main agent is located in the leftmost cell of the block in an
initial state that we call $s_{1}$. At each site where the data layer
contains a data bit from the current word, there are two data agents,
one of each type Left and one of type Right (not to be confused with
their motion state!), with motion symbol {}``stationary'' and empty
data.

The main agent is initially in a state $s_{1}$. While in this state
he moves one cell to the right at each time step, and at each cell
has the following effect: the data bit from the data layer is copied
to the data agents, and the data agents are {}``launched'', i.e.
their motion symbol is set as appropriate (the motion state of the
L:eft agent is set to {}``left'', and that of the other to {}``right'').
The end of the current data word is indicated by a roadsign (special
symbol in the block layer), and when reached it causes the main agent
to enter internal state $s_{2}$. 

At this point there are two new copies of the current data word, encoded
in sequences of agents who are marching left and right at unit speed,
forming what we shall call {}``caravans''. We would like the main
agent to meet the caravans arriving in its block and cause them to
stop. Thus he should do {}``nothing'' for a while, and arrive at
a designated spot at a designated time to receive the first caravan.
Since entering a state of inactivity and staying there is not an invertible
operation we instead have the agent walk to the right while in state
$s_{2}$ until it reaches a specially placed roadsign. At this point
it enters state $s'_{2}$, walks left until it reaches another designated
roadsign, and enters state $s_{3}$. Since the transitions are controled
by encounters with specific roadsigns they are invertible, and by
by controlling the positions of the roadsigns we can determine the
time and place at which the agent enters state $s_{3}$. 

We assume that the main agent enters state $s_{3}$ just as the first
data agent in caravan of data agents from the block to the right is
arriving at the cell where it is to stop. In state $s_{3}$ the main
agent walks to the right, and whenever it shares a left-moving data
agent it sets its motion symbol to {}``stationary''. Thus the main
agent will cause the left-moving caravan to halt. A roadsign indicates
to the main agent that it has reached the last data agent in the caravan
(this position is completely determined by the parameters and the
choice of roadsigns so far); when this roadsign is reached the main
agent enters state $s_{4}$.

In state $s_{4}$ the main agent behaves similarly, moving left and
stopping data agents arriving from the right. When the last data agent
is halted, a roadsign forces the main agent into a new state $s_{5}$.

At this point we have completed our goal: the state layer of each
block contains data words from the current block and its two neighbors.
The time elapsed from the beginning of the cycle is $B+O_{N}(1)$;
the $B$ term is the time it actually takes each data bit to travel,
and the $O_{N}(1)$ is an adjustment encompassing overhead and the
fact that slightly less than $B$ may have been traveled (we have
some choice about where to halt the caravans).

With regard invertibility, note that given the state $s_{i}$ in which
the main agent is found, the effect on data agents is invertible,
and the transitions between $s_{i}$ to $s_{i+1}$ is controlled by
roadsigns, and is invertible as well.

We conclude this outline with some further comments.

In our example we have not demontrated the deletion of bits from the
data or state layers. To make this invertible each deleted bit must
be present in some other form in the same cell. For example, when
transferring the current data word to the state layer the main agent
will delete a bit from the state layer while at the same time recording
it in a data agent; thus this operation can be reversed.

The next step in the cycle is to transfer the data-agents, representing
three data-words, to the right, stopping opposite each corresponding
triple in the program layer. This is similar to the transfer we just
performed, except there is no main agent on the receiving end to halt
the caravans. This is solved by positioning roadsigns that cause the
data agents to reverse direction. These act as {}``reflecting walls'',
and allow the main agent who launched the caravan to also halt them
(the order of bits in the data words is reversed; one can either work
with this reversal, or perform the reflection twice).

Another point that needs some care is the comparison stage, at which
the three data words represented by the data agents are compared to
the corresponding words in the program layer. When comparing the input
triple in the program to the corresponding triple in the state layer,
one cannot simply traverse them both from left to right, say, and
take note if they differ at some point, because this is not reversible
(when going backwards, you are in a state of knowing that there is
a differing pair of symbols until you reach the leftmost such pair.
After that your state is that of not yet having seen a difference.
But there is no way to know, when you reach a differing pair, if it
is the leftmost such pair or not). To overcome this, one begins on
the left, say, and puts down markings in the state layer: green until
the first difference, if there is one, and red thereafter. When the
end of the comparison is reached we note the current color, and then
go back, right to left, and erase the color markings. This procedure
is invertible with range $1$. 

Lastly, we discuss how to control the number of iterations of the
{}``shifting'' stage, which must occur $D$ and $W$ times. Perhaps
the simplest is that, during the time that the data is being transferred,
the main agent {}``counts down''. This can be done in the simplest
of ways, by coding $D$ and $W$ into the program layer as a sequence
of $D$ or $W$ special symbols, respectively; transferring them to
the state layer (as stationary data agents); and {}``crossing one
out'', i.e. resetting one of them, with each {}``shift'' iteration.

\section{\label{sec:Realizing-unique-non-expansive-directions}Realizing unique
non-expansive directions}

Before continuing, let us make some observations about the construction
above. The following is clear from the construction:
\begin{lem}
\label{lem:monotonicity-in-Y}Suppose $X_{1}\subseteq X$ is constructed
from the parameters $Y_{1},\varphi,N,B,W,D$. Suppose that $Y_{2}\subseteq Y_{1}$
is $\sigma$- and $\varphi$-invariant, and let $X_{2}$ be constructed
using parameters $Y_{2},\varphi,N,B,W,D$. Then $X_{2}\subseteq X_{1}$.
\end{lem}

\begin{lem*}
$\widehat{\pi}|_{X}=(\pi|_{X})^{-1}$\end{lem*}
\begin{proof}
Recall that $X=\overline{\cup X'}$, the union being over all systems
constructed from permissible parameters. The lemma follows from the
fact that for each such $X'$ we have $\widehat{\pi}|_{X'}=(\pi|_{X'})^{-1}$,
and in all cases $\pi$ is given by the same block code.
\end{proof}
Next, we relate the prediction shapes of $\varphi|_{Y}$ and $\pi|_{X'}$. 
\begin{lem}
\label{lem:shape-contraction}Suppose $X'\subseteq X$ is constructed
from the parameters $Y,\varphi,N,B,W,D$. Let $\Lambda$ be a prediction
shape for $\varphi|_{Y}$. Then $A(\Lambda)$ is a prediction shape
for $\pi|_{X'}$, where $A:\mathbb{R}^{2}\rightarrow\mathbb{R}^{2}$
is the linear map fixing $e_{1}=(1,0)^{T}$ and mapping $e_{2}=(0,1)^{T}$
to the vector $(D,T/B)$, or in matrix form,\[
A=\left[\begin{array}{cc}
1 & D\\
0 & T/B\end{array}\right]\]
\end{lem}
\begin{proof}
This is immediate from the identification\[
(X',\sigma,\pi)\cong(Y\times\{0,\ldots,B-1\}\times\{0,\ldots,T-1\},\sigma_{B},\varphi_{T})\qedhere\]

\end{proof}
We now undertake the main construction of this section, and proceed
to analyze it. Let $N$ be the number of symbols in the alphabet of
$X$. Fix a sequence $B_{n},D_{n},W_{n}$ of parameters for the construction
above, with $W_{k}\geq2$ (hence $T_{k}/B_{k}\geq2$). 

For a subshift $Y\subseteq X$ invariant under $\pi$, let $Z_{n}(Y)$
denote the system $X'$ constructed above with parameters $N,Y,\pi,B_{n},D_{n},W_{n}$.
Define \[
Z_{n}^{n+k}=Z_{n}(Z_{n+1}(\ldots(Z_{n+k}(X))\ldots)).\]
An induction using lemma \ref{lem:monotonicity-in-Y} shows that $Z_{n}^{n+k+1}\subseteq Z_{n}^{n+k}$;
thus \[
Z_{n}^{\infty}=\bigcap_{k=1}^{\infty}Z_{n}^{n+k}\]
is non-empty and $\sigma$- and $\pi$-invariant. Finally, set\[
Z=Z_{1}^{\infty}.\]
It is easy to see that we have the relation \begin{equation}
Z=Z_{1}^{\infty}=Z_{1}(Z_{2}^{\infty})=Z_{1}(Z_{2}(Z_{3}^{\infty}))=\ldots\label{eq:Z-recursion}\end{equation}
etc. 

We will now show that there is a (necessarily unique) line $\ell$
through the origin so that \[
\Lambda=\ell^{1}=\{u\in\mathbb{R}^{2}\,:\, d(u,\ell)<1\}\]
is a prediction shape of $Z$, and calculate the slope of $\ell$.
It suffices to write $\Lambda$ as an increasing union of prediction
shapes for $\pi|Z$.

Let $\Delta$ denote the unit ball in $\mathbb{R}^{2}$ with the norm
$\left\Vert \cdot\right\Vert _{1}$, which is a prediction shape for
$\pi$ and any $\pi$-invariant subshift of $X$. Let $A_{n}:\mathbb{R}^{2}\rightarrow\mathbb{R}^{2}$
denote the map associated as in lemma \ref{lem:shape-contraction}
to $Z_{n}(\cdot)$. Since $Z_{n}^{\infty}=Z_{n}(Z_{n+1}^{\infty})$
and $\Delta$ is a prediction shape for $Z_{n+1}^{\infty}$, it follows
from lemma \ref{lem:shape-contraction} that $A_{n}(\Delta)$ is a
prediction shape for $Z_{n}^{\infty}$. Therefore since $Z_{n-1}^{\infty}=Z_{n-1}(Z_{n}^{\infty})$
the same lemma gives that $A_{n-1}(A_{n}\Delta)$ is a prediction
shape for $Z_{n-1}^{\infty}$, and iterating we have that\[
\Delta_{n}:=A_{1}A_{2}\ldots A_{n}(\Delta)\]
is a prediction shape for $Z=Z_{1}^{\infty}$. 

Notice that each the shape $\Delta_{n}$ is a quadrilateral having
two vertices, $(-1,0)$ and $(1,0)$, on the $x$-axis, one vertex
above the $x$-axis, and one below it. We shall now analyze the asymptotic
behavior of these vertices. Let $A$ be the matrix in lemma \ref{lem:shape-contraction},
i.e. it is of the form \[
A=\left[\begin{array}{cc}
1 & D\\
0 & T/B\end{array}\right]\]
for integer parameters $D,T,B$ and $T/B$. Let $(x,y)^{T}\in\mathbb{R}^{2}$
with $y>0$ and $(x',y')^{T}=A(x,y)$. Then \[
\frac{x'}{y'}=\frac{DB}{T}+\frac{B}{T}\cdot\frac{x}{y}\]

Now fix $n$ and let $(x_{n},y_{n})^{T}$ denote the vertex of the
quadrilateral $\Delta_{n}$ that lies in the upper half plane. Then\[
(x_{n},y_{n})^{T}=A_{1}\ldots A_{n}(0,1)\]
so $x_{n}/y_{n}$ is given by \[
\frac{x_{n}}{y_{n}}=\frac{D_{1}B_{1}}{T_{1}}+\frac{B_{1}}{T_{1}}\left(\frac{D_{2}B_{2}}{T_{2}}+\frac{B_{2}}{T_{2}}\left(\ldots\left(\frac{D_{n}B_{n}}{T_{n}}+\frac{B_{n}}{T_{n}}\cdot\frac{0}{1}\right)\ldots\right)\right)\]
or, using (\ref{eq:T-bound}) and writing \[
\alpha_{k}=\frac{D_{k}B_{k}}{T_{k}}=(1+\varepsilon_{k})\frac{D_{k}}{|D_{k}|+W_{k}}\qquad\mbox{and}\qquad\beta_{k}=\frac{B_{k}}{T_{k}}=\frac{1}{(1+\varepsilon_{k})(|D_{k}|+W_{k})}\]
where $\varepsilon_{k}=O_{N,\varphi}(\frac{1}{B_{k}})$, we have\[
\frac{x_{n}}{y_{n}}=\alpha_{1}+\beta_{1}(\alpha_{2}+\beta_{2}(\ldots(\alpha_{n}+\beta_{n}\cdot0))\ldots)\]
Since $T_{k}/B_{k}\geq2$ for all $k$ by our choice of parameters,
$\beta_{k}<1/2$, so this sequence converges to some $\lambda\in(-1,1)$.
For the same reason we have $y_{n}\geq2^{n}$ and thus for large enough
$n$ we have $x_{n}\rightarrow\infty$. Since the slopes of the two
sides of $\Delta_{n}$ which lie in the upper half plane are\[
\frac{y_{n}}{x_{n}-1}\mbox{ and }\frac{y_{n}}{x_{n}+1}\]
it follows that that these slopes converge to the common value $1/\lambda$.

A similar calculation shows that as $n\rightarrow\infty$ the remaining
vertex of $\Delta_{n}$ grows unboundedly in norm and that the slope
of the remaining two sides converges to $1/\lambda$ as well. 

Finally, it follows that there is an increasing subsequence of the
$\Delta_{k}$'s whose union is (necessarily) the set $\Lambda=\ell^{1}$,
where $\ell$ is the line with slope $1/\lambda$. 

In summary, we have:
\begin{thm}
\label{thm:1-over-lambda-is-unique-nonexpansive-direction}$(Z,\left\langle \sigma,\pi\right\rangle )$
has a unique non-expansive direction whose slope is $1/\lambda$,
where \[
\lambda=\alpha_{1}+\beta_{1}(\alpha_{2}+\beta_{2}(\alpha_{3}+\ldots(\alpha_{n}+\beta_{n}(\ldots)))\ldots)\]
and $\alpha_{k},\beta_{k},\varepsilon_{k}$ are as given above. \end{thm}
\begin{proof}
Theorem \ref{thm:speeds-and-non-expansive-directions} implies that
any line with slope different from $1/\lambda$ is expansive (and
the horizontal direction is as well, since $\sigma$ acts expansively
by definition). That the line with slope $1/\lambda$ is a non-expansive
direction then follows from the fact that any infinite system has
non-expansive directions, and the following claim.\end{proof}
\begin{claim}
\label{cla:Z-is-infinite}For any choice of parameters, the system
$Z$ is infinite.\end{claim}
\begin{proof}
From the representation (\ref{eq:simulation-isom}), we see that every
choice of $z_{2}\in Z_{2}^{\infty}$ is represented, modulo $\sigma$-shifts,
by $T_{1}B_{1}$ points $z_{1}=z_{1}(z_{2})\in Z$. Thus $|Z|=|Z_{1}^{\infty}|\geq T_{1}B_{1}|Z_{2}^{\infty}|$.
Similarly, $|Z_{2}^{\infty}|\geq T_{2}B_{2}|Z_{3}^{\infty}|$, so
$|Z|\geq T_{1}B_{1}T_{2}B_{2}|Z_{3}^{\infty}|$; and so one. Since
$T_{n}B_{n}\geq2$ for each $n$, the conclusion follows.
\end{proof}
We have shown that out construction yields systems with a unique non-expansive
direction of the form $1/\lambda$. It remains to show that any direction
can be attained.  The following an elementary exercise in representing
reals:
\begin{lem}
\label{lem:any-value-of-lambda-is-possible}For any real number $|\theta|\leq1$
occurs as the number $\lambda$ for some choice of the parameters
$B_{k},W_{k},D_{k}$, and we may choose $W_{k}\geq2$.\end{lem}
\begin{proof}
We assume for convenience that $\theta\geq0$; the cast $\theta<0$
follows similarly, the only difference being that all the $D$'s are
then negative and the endpoints of segments must appear in the reverse
order.

For some pair of integers $W\geq1$ and $D\geq0$, the number $\theta$
lies between $\frac{D}{|D|+W+1}$ and $\frac{D+1}{|D|+W+1}$. Taking
$W_{1}=W,D_{1}=D$ and taking $B_{1}$ to be sufficiently large, we
can make the error $\varepsilon_{k}$ arbitrarily small and obtain
\[
\theta\in(\alpha_{1},\alpha_{1}+\beta_{1})\]
One proceeds inductively to choose $B_{2},W_{2},D_{2}$ so that\[
\frac{\theta-\alpha_{1}}{\beta_{1}}\in(\alpha_{2},\alpha_{2}+\beta_{2})\]
implying that\[
\theta\in(\alpha_{1}+\beta_{1}\alpha_{2},\alpha_{1}+\beta_{1}(\alpha_{2}+\beta_{2}))\]
and so on (note that the closure of each interval is in the interior
of the previous one).
\end{proof}
Theorem \ref{thm:1-over-lambda-is-unique-nonexpansive-direction}
and lemma \ref{lem:any-value-of-lambda-is-possible} show that any
line with slope $\theta$, $|\theta|>1$, occurs as the unique non-expansive
direction of some action. All other directions can be attained from
this result by re-parametrizing the acting group. This completes the
proof of the main part of theorem \ref{thm:main}.

It remains only to show that in the case of a rational direction the
action in that direction is non-trivial. This is shown in the same
way as the proof of claim \ref{cla:Z-is-infinite}, that $Z$ is infinite;
we omit the details.

\section{\label{sec:Lyapunov-exponents}Lyapunov exponents}

Given a subshift $Y$ and an endomorphism $\varphi:Y\rightarrow Y$,
Shereshevsky \cite{Shereshevsky92} defined the Lyapunov exponents
$\lambda^{+},\lambda^{-}$ as follows. For $y\in Y$ define \[
I_{t}^{+}(y)=\min\left\{ n\,:\,\forall z\in Y\;\forall0\leq s\leq t\quad\left(z|_{[-n,\infty)}=y|_{[-n,\infty)}\quad\implies\quad(\varphi^{s}z)|_{[0,\infty)}=(\varphi^{s}y)|_{[0,\infty)}\right)\right\} \]
and similarly\[
I_{t}^{-}(y)=\min\left\{ n\,:\,\forall z\in Y\;\forall0\leq s\leq t\quad\left(z|_{(-\infty,n]}=y|_{(-\infty,n]}\quad\implies\quad(\varphi^{t}z)|_{(-\infty,0]}=(\varphi^{t}y)|_{(-\infty,0]}\right)\right\} \]
Set \[
\Lambda_{t}^{\pm}=\max_{y\in Y}\max_{i\in\mathbb{Z}}I_{t}^{\pm}(\sigma^{i}y)\]
The Lyapunov exponents are then defined by \[
\lambda^{\pm}=\liminf_{t\rightarrow\infty}\frac{\Lambda_{t}^{\pm}}{t}\]
 (Shereshevsky's original definition of $\lambda^{\pm}$ differs from
the above but is equivalent by \cite{Tisseur2000}). 

We omit the proof of the following, which is an immediate consequence
of the definitions:
\begin{prop}
$\varphi:Y\rightarrow Y$ be an automorphism of an infinite subshift
$Y$ and let $\Lambda$ be a prediction shape for $\varphi|_{Y}$.
Let $\theta^{+},\theta^{-}$ denote the (possibly infinite) slopes
of the right- and left-tangent rays to $\partial\Lambda$ at $(-1,0)$
and $(1,0)$, respectively. Then $\lambda^{+}\leq1/\theta^{+}$ and
$\lambda^{-}\leq-1/\theta^{-}$.\end{prop}
\begin{cor}
If the strip $\Lambda=\{(x,y)\,:\,|x|<1\}$ is a prediction shape
for $\varphi|_{Y}$, then $\lambda^{+}=\lambda^{-}=0$. 
\end{cor}
Tisseur \cite{Tisseur2000} and later Tisseur and Bressaud \cite{BressaudTisseur2007}
studied the relation between Lyapunov exponents, particularly the
case of zero Lyapunov exponent, and the existence of equicontinuity
points for the action of $\varphi$. Let us recall some definitions.
For an endomorphism $\varphi$ acting on a subshift $Y\subseteq\Sigma^{\mathbb{Z}}$,
we say that a finite word $a\in\Sigma^{n}$ is a \emph{blocking word
}if, for any pair $y,z\in Y$ with $y|_{[1,n]}=z|_{[1,n]}=a$ and
$y|_{[1,\infty)}=z|_{[1,\infty)}$, we also have $(\varphi^{t}y)|_{[1,\infty)}=(\varphi^{t}z)|_{[1,\infty)}$
for all $t$; and also for any $y,z$ satisfying $y|_{[-n,-1]}=z|_{[-n,-1]}=a$
and $y|_{(-\infty,-1]}=z|_{(-\infty,-1]}$ we have $(\varphi^{t}y)|_{(-\infty,-1]}=(\varphi^{t}z)|_{(-\infty,-1]}$.
The condition that $\varphi$ have equicontinuity points is equivalent
to $\varphi$ having a blocking word. Also, not having equicontinuity
points is equivalent to $\varphi$ acting on $Y$ with sensitive dependence
on initial conditions.

Returning to the matter at hand, Bressaud and Tisseur conjectured
that when $\varphi$ is a cellular automaton acting sensitively (i.e.
without equicontinuity points) on $Y$, some point $y\in Y$ has $\liminf_{t\rightarrow\infty}I_{t}^{+}(y)>0$
or $\liminf_{t\rightarrow\infty}I_{t}^{-}(y)>0$ (\cite[Conjecture 3]{BressaudTisseur2007}).
Our construction provides a counterexample for this conjecture, as
we describe next. We continue to use the notation introduced in the
previous section during the construction of $Z$. 
\begin{prop}
\label{pro:Z-is-counterexample}The action of $\pi$ on $Z$ does
not have equicontinuity points.\end{prop}
\begin{proof}
Suppose that $x$ were an equicontinuity point. Then there is an $n$
so that if $x(i)=y(i)$ for $|i|\leq n$ and $y\in Z$ then $\pi^{t}x(0)=\pi^{t}y(0)$
for all $t$. We show this is impossible.

Consider the block $a=x|_{[-1,1]}$. This block contains certain information
about the configuration of the one or two blocks in $Z$ which intersect
the coordinates $[-n,n]$. This in turn may encode some information
about the location and state of the one or two blocks in $Z_{2}^{\infty}$
whose state are encoded in the two blocks in the first level of $Z$;
and so on. But since $a$ is a finite block, there must be a $k$
so that $a$ contains no information about the blocks whose encoding
it intersects in $Z_{k}^{\infty}$. Now note that we can choose $y$
so that the state of these $Z_{k}^{\infty}$-blocks in $y$ differs
from their state in $x$. Thus there is some bit in these $Z_{k}^{\infty}$-blocks
that differs. This bit will eventually be transported across $a$;
therefore there will be a $t$ so that $\pi^{t}y(0)\neq\pi^{t}x(0)$.
\end{proof}
This last proposition may seem surprising, since the dynamics of $(Z,\pi)$
appear at first glance to be almost periodic. However, it is not in
reality so.

\section{\label{sec:alternative-construction}An alternative construction
in the rational case}

We present here a construction that arose in discussions with Doug
Lind and provides a simpler example of a system whose unique non-expansive
direction is the vertical axis, but no power of the action in that
direction is the identity (though in other ways the action is dynamically
rather trivial). This example is significantly simpler than the one
above and may be adapted to give examples in other rational directions,
but we have been unable to get any irrational direction with this
method. It is striking to us that the irrational case is so much more
difficult than the rational one, and it would be interesting if a
simpler construction for that case were found. 

As before, we construct a subshift $X\subseteq\Sigma^{\mathbb{Z}}$
and an automorphism $\pi:X\rightarrow X$ such that the vertical strip
of width $2$ around the $y$-axis is a prediction shape for $\pi$. 

Fix a parameter $n\in\mathbb{N}$. The alphabet $\Sigma$ consists
of the symbols\begin{eqnarray*}
- &  & \mbox{(blank)}\\
\rightarrow,\leftarrow &  & \mbox{(arrows)}\\
\overset{k}{[},\overset{k}{]} & \mbox{for }0\leq k\leq n & \mbox{(brackets, with counter }k\mbox{)}\\
\overset{k}{\underset{*}{[}},\overset{k}{\underset{*}{]}} & \mbox{for }0\leq k\leq n-1 & \mbox{(marked brackets, with counter }k\mbox{)}\end{eqnarray*}
Each configuration in $X$ will have at most one arrow symbol in it;
the rest will be blanks and brackets. Adjacent brackets will not be
allowed, instead between any pair of brackets there will always be
at least one blank or arrow symbol. 

We shall later describe the configurations of $X$ in more detail,
but first we define the automorphism $\pi$ by giving the relevant
transitions, from which a range-2 block code may be derived. The transitions
are\begin{eqnarray}
\rightarrow- & \mbox{becomes} & -\rightarrow\label{eq:joe-1}\\
\rightarrow\overset{n}{[}- & \mbox{becomes} & -\overset{n-1}{\underset{*}{[}}\rightarrow\label{eq:joe-2}\\
\rightarrow\overset{k}{]}- & \mbox{becomes} & \leftarrow\overset{k-1}{]}-\mbox{ if }k>0\label{eq:joe-3}\\
\rightarrow\overset{0}{]}- & \mbox{becomes} & -\overset{n}{]}\rightarrow\label{eq:joe-4}\\
-\overset{k}{\underset{*}{[}}\leftarrow & \mbox{becomse} & -\overset{k-1}{\underset{*}{[}}\rightarrow\mbox{ if }k>0\label{eq:joe-5}\\
-\overset{0}{\underset{*}{[}}\leftarrow & \mbox{becomes} & -\overset{n}{[}\rightarrow\label{eq:joe-6}\end{eqnarray}
together with the symmetric rules obtained by reversing left and right,
e.g. $-\overset{0}{[}\leftarrow$ becomes $\leftarrow\overset{n}{[}-$
(reversal of \eqref{eq:joe-4}).

To interpret this, one may imagine that the arrow represents an agent
walking in a landscape of brackets, proceeding in the direction the
arrow points to (rule \eqref{eq:joe-1}). The behavior of the agent
when it encounters a bracket depends on the orientation and the data
on the bracket. 
\begin{itemize}
\item When the agent approaches a bracket from the {}``inside'',

\begin{itemize}
\item If the bracket has positive counter, the agent decrements the counter
by $1$ and turns around (rule \eqref{eq:joe-3},\eqref{eq:joe-5}). 
\item If the counter is $0$ then it is reset to $n$; next, if the bracket
was marked the agent removes the mark and turns around, but if it
was unmarked the agent passes through (rule \eqref{eq:joe-4},\eqref{eq:joe-6}).
\end{itemize}
\item When the agent approaches a bracket from the {}``outside'' it marks
the bracket, decrements the counter, and passes through (rule \ref{eq:joe-2}).
We shall arrange that such an encounter only occurs when the counter
is $n$ and the bracket is unmarked. 
\end{itemize}
For example, starting from the pattern $\rightarrow\overset{n}{[}-\overset{n}{]}-$
the agent will {}``enter'' the region between the brackets, reverse
its direction $2n$ times, and emerge from the right side. Note that
upon its exit it leaves behind the configuration as he found it, i.e.
the final pattern is $-\overset{n}{[}-\overset{n}{]}\rightarrow$. 

We next describe the allowable arrangement of brackets in $X$. We
first define special sequences of {}``plain'' brackets $[,]$ which
are arranged in a hierarchical manner. Begin by choosing a periodic
subset $I_{1}\subseteq\mathbb{Z}$ of period $2$ (there are two ways
to do this), and set the symbols $y_{i},i\in I$ to be alternately
$[$ and $]$ (this can again be done in two ways). Half of the symbols
in $\mathbb{Z}\setminus I_{1}$ are now trapped between matching brackets;
let $I'_{2}$ denote the half which is not, which is a coset of $4\mathbb{Z}$.
Next, choose a subset $I_{2}\subseteq I'_{2}$ of relative period
2 (a coset of $8\mathbb{Z}$; there are two choices) and define $y_{i},i\in I_{2}$
to be alternately $[$ and $]$ (again two choices). Let $I'_{3}\subseteq\mathbb{Z}\setminus(I_{1}\cup I_{2})$
be those indices which have not yet been determined, and which are
not trapped between matching brackets, and choose $I_{3}\subseteq I'_{3}$
a subset of relative period $2$ (a coset of $32\mathbb{Z}$). Proceed
in this manner to define $y_{i}$ for $i\in I_{3}$ and $I'_{4},I_{4}$,
etc. After carrying this out for all $n$ every $i\in\mathbb{Z}$,
with possibly one exception, is trapped between some pair of brackets;
the remaining point, if it exists, may be left blank or given the
symbol $[$ or $]$. Three steps in the construction of such a $y$
appears below (big brackets indicate the addition at each stage).\[
\begin{array}{c}
-[-]-[-]-[-]-[-]-[-]-[-]-[-]-[-]-[-]-[-]-[-]-[-]-[-]-\\
\bigg]\,[-]-[-]\;\bigg[\;[-]-[-]\;\bigg]\;[-]-[-]\;\bigg[\;[-]-[-]\;\bigg]\;[-]-[-]\;\bigg[\;[-]-[-]\;\bigg]\;[-]-\\
]\;[-]\;\bigg[\;[-]\;[\;[-]-[-]\;\,]\;\,[-]-[-]\;\,[\;\,[-]-[-]\;\,]\;\,[-]\;\bigg]\;[-]\;[\;[-]-[-]\;]\;[-]-\end{array}\]

We define a pre-block to be a subword of an hierarchical arrangement
as above, which consists of a matched pair of brackets and the region
between them. The pattern $[-]$ is a pre-block, and we call it the
level-0 pre-block; next is $[[-]-[-]]$, a level-1 pre-block; in general,
a pre-block containing level-$n$ pre-blocks but no level-$(n+1)$
pre-block is a level-$(n+1)$ pre-block.

Next, define a level-$n$ block to be the word obtained from a level-$n$
pre-block by inserting a blank in between every pair of symbols. 

We now define the admissible words in $X$. Let $a$ be a block, and
consider the patterns $\rightarrow a-$ and $-a\leftarrow$, which
we extend with blanks in both directions (but we suppress these blanks
notationally). It is easy to verify that after finitely many iterations
of $\pi$ we get the patterns the $-a\rightarrow$ and $\leftarrow a-$,
respectively. Let $L(a)$ denote the set of intermediate patterns
obtained in this way. We define $X$ to be the subshift such that
every finite word in $X$ appears as a subword of some $b\in L(a)$,
for some block $a$. 

It is not hard to check that for each $x\in X$ there is a coset of
$2\mathbb{Z}$ on which there appears a hierarchical configuration
of brackets in the sense above, with brackets now carrying counters
and markings. On the complementary coset there appear only blanks
and possibly an arrow. One can also verify that if $x\in X$ contains
a block that does not contain an arrow, then that block consists of
unmarked brackets with counters equal to $n$. 

The point of the construction is the following. The changes that occur
in a configuration under $\pi$ occur only at the site of an arrow
or adjacent to an arrow, so in order to understand the propagation
of perturbations to a configuration under $\pi$ we must understand
is the rate at which the arrow moves. For this, note that in order
to {}``pass through'' the block \[
a_{1}=\overset{n}{[}---\overset{n}{]}\]
 requires $6n$ steps. Now consider the block \[
a_{2}=\overset{n}{[}-\overset{n}{[}---\overset{n}{]}---\overset{n}{[}---\overset{n}{]}-\overset{n}{]}\]
To pass through this requires the arrow to go back and forth $2n$
times between the external brackets; each time it must cross the inner
two brackets twice, taking $6n$ steps each time. Thus to cross $a_{2}$
requires $2n\cdot(5+2\cdot6n)$. 

Continuing in this way, one may show that the time to cross the level-$n$
block $a_{k}$ is $(c_{n})^{k}$, where $c_{n}\rightarrow\infty$
with $n$. On the other hand, the width of $a_{k}$ is $d^{k}$ for
a constant $d$ independent of $n$, and the blocks $a_{k}$ appear
periodically with period $d^{k}$ in any configuration of $X$. It
follows that in order to travel a distance of $Nd^{k}$ will require
time on the order of $N(c_{n})^{k}$, i.e. over large scales the rate
of travel is logarithmic (we remark that this is the slowest possible
rate; if the rate were sub-logarithmic we would have, counting configurations,
that the action of $\pi$ were periodic).

In particular, if we know the configuration $x|_{[-N,N]}$ for $x\in X$
we can predict $x|_{[-N+\log N,N-\log N]}$ up to time $O(N)$. It
follows that the vertical strip of width $2$ is a prediction shape
for $X$, as desired. We omit the details.

Finally, since the arrow does travel arbitrarily far in some configurations
(in fact, in any configuration containing the arrow), it follows that
the action of $\pi$ is not periodic. We remark, however, that the
dynamics of $\pi$ are in other ways rather trivial, e.g. all invariant
measures are concentrated on fixed points, and there are uncountable
many of these (the configurations without an arrow). 

\bibliographystyle{plain}
\bibliography{non-expansive-directions}

\begin{thebibliography}{10}

\bibitem{AlbertCulik1987}
J{\"u}rgen Albert and Karel Culik, II.
\newblock A simple universal cellular automaton and its one-way and totalistic
  version.
\newblock {\em Complex Systems}, 1(1):1--16, 1987.

\bibitem{Boyle2008}
Mike Boyle.
\newblock Open problems in symbolic dynamics.
\newblock {\em to appear in Contemporary Mathematics}, 2008.
\newblock http://www.math.umd.edu/~mmb/papers/openfinalsub2nov2008.pdf.

\bibitem{BoyleLind97}
Mike Boyle and Douglas Lind.
\newblock Expansive subdynamics.
\newblock {\em Trans. Amer. Math. Soc.}, 349(1):55--102, 1997.

\bibitem{BressaudTisseur2007}
Xavier Bressaud and Pierre Tisseur.
\newblock On a zero speed sensitive cellular automaton.
\newblock {\em Nonlinearity}, 20(1):1--19, 2007.

\bibitem{Gacs2001}
Peter G{\'a}cs.
\newblock Reliable cellular automata with self-organization.
\newblock {\em J. Statist. Phys.}, 103(1-2):45--267, 2001.

\bibitem{Madden2000}
K.~M. Madden.
\newblock A single nonexpansive, nonperiodic rational direction.
\newblock {\em Complex Systems}, 12(2):253--260, 2000.

\bibitem{MoritaHarao89}
K~Morita and M~Harao.
\newblock Computation universality of one-dimensional reversible (injective)
  cellular automata.
\newblock {\em IEICE Trans. Inf. \& Syst.}, 72:758--762, 1989.

\bibitem{Shereshevsky92}
M.~A. Shereshevsky.
\newblock Lyapunov exponents for one-dimensional cellular automata.
\newblock {\em J. Nonlinear Sci.}, 2(1):1--8, 1992.

\bibitem{Tisseur2000}
P.~Tisseur.
\newblock Cellular automata and {L}yapunov exponents.
\newblock {\em Nonlinearity}, 13(5):1547--1560, 2000.

\bibitem{Walters82}
Peter Walters.
\newblock {\em An introduction to ergodic theory}, volume~79 of {\em Graduate
  Texts in Mathematics}.
\newblock Springer-Verlag, New York, 1982.

\end{thebibliography}

\end{document}